\documentclass[a4paper,11pt,oneside,reqno]{amsart}
\usepackage{amsmath,amsfonts,amstext,amssymb,amsbsy,amsopn,amsthm,eucal,amsrefs}
\usepackage{vmargin}
\usepackage[pdftex]{graphicx}

\usepackage{amssymb}
\usepackage{mathrsfs}
\usepackage{stmaryrd}
\usepackage{amsmath}
\usepackage{amsthm}
\usepackage{enumerate}
\usepackage[usenames,dvipsnames]{xcolor}
\usepackage{nomencl}
\usepackage[bookmarks=true]{hyperref}   
\hypersetup{
    colorlinks = red,
    citecolor=Mahogany,
    filecolor= black,
    linkcolor=Plum,
    urlcolor=Aquamarine
}

\usepackage{setspace}

\usepackage{times}

\usepackage[T1]{fontenc}

\usepackage[active]{srcltx}

\newtheorem{innercustomthm}{Theorem}
\newenvironment{customthm}[1]
  {\renewcommand\theinnercustomthm{#1}\innercustomthm}
  {\endinnercustomthm}

\usepackage{enumitem}

\makeatletter
\renewcommand\section{\@startsection{section}{1}{\z@}%
                       {-3\p@ \@plus -4\p@ \@minus -4\p@}%
                       {3\p@ \@plus 4\p@ \@minus 4\p@}%
                      {\normalfont\normalsize\centering\scshape}}
\makeatother

\hypersetup{
    colorlinks,%
}

\author{Sajjad Lakzian \textsuperscript{\textdagger}}
\author{Zachary Mcguirk \textsuperscript{\ddag}}

\date{\today}

\address{\textsuperscript{\textdagger} \small Sajjad Lakzian,
Mathematics Department, Fordham University, Rose Hill Campus} 
\urladdr{\href{https://sites.google.com/site/sajjadlakzianmath/}{https://sites.google.com/site/sajjadlakzianmath/}}
\email{\href{mailto:lakzians@gmail.com}{lakzians@gmail.com}}

\address{\textsuperscript{\ddag} \small Zachary McGuirk,
	Mathematics Department, Graduate Center, CUNY} 
\email{\href{mailto:zmcguirk@gradcenter.cuny.edu}{zmcguirk@gradcenter.cuny.edu}}

\thanks{\textbf{\textit{The first author is supported by the PMC Visiting Assistant Professorship at Fordham University}}}


\newcommand{\RNum}[1]{\uppercase\expandafter{\romannumeral #1\relax}}

\usepackage{fancyhdr}
\fancypagestyle{plain}{
    \fancyhf{} 
    \fancyfoot[C]{\thepage} 
    \fancyhead[LO]{\small Conical Curvature-Dimension and Global Poincar\'e Inequality on Graphs.}
    \fancyhead[RO]{\small }
}

\def\colour{\colour}

\def\colour{\color}

\newtheorem{theorem}{Theorem}[section]
\newtheorem{corollary}[theorem]{Corollary}
\newtheorem{lemma}[theorem]{Lemma}

\newtheorem{definition}[theorem]{Definition}
\newtheorem{remark}[theorem]{Remark}

\newtheorem{thm}{Theorem}[section]

\theoremstyle{definition}

\newtheorem{defn}[thm]{Definition}

\newcommand{\be}{\begin{equation}}
\newcommand{\ee}{\end{equation}}

\newcommand{\R}{\mathbb{R}}

\renewcommand{\epsilon}{\varepsilon}

\renewcommand{\phi}{\varphi}

\newcommand{\Ric}{{\rm Ric}}			

\DeclareFontFamily{OT1}{restrictfont}{}
\DeclareFontShape{OT1}{restrictfont}{m}{n}{<-> fmvr8x}{}

\begin{document}
	
	\title{A Global Poincar\'e inequality on Graphs via a Conical Curvature-Dimension Condition}
	\date{\today}

	\maketitle
	
	\renewcommand\abstractname{\footnotesize \textbf{ABSTRACT}}
	\rule{\textwidth}{1px}\\
	\begin{abstract}
		We introduce and study the \emph{conical curvature-dimension condition}, $CCD(K,N)$, for graphs. We show that $CCD(K,N)$ provides necessary and sufficient conditions for the underlying graph to satisfy a sharp global Poincar\'e inequality which in turn translates to a sharp lower bound for the first eigenvalues of these graphs. Another application of the \emph{conical curvature-dimension} analysis is finding a sharp estimate on the curvature of complete graphs. 
	\end{abstract}
	\rule{\textwidth}{0.5px}
	
	

	\parindent0cm
	\setlength{\parskip}{\baselineskip}
	
	\setcounter{tocdepth}{1}
	\small
	\tableofcontents
	\normalsize
	\addtocontents{toc}{~\hfill\textbf{Page}\par}
\section{Introduction}
  The relation between Ricci curvature bounds and the analytic and geometric properties of a smooth Riemannian manifold is a well studied subject in geometric analysis. Thanks to the seminal work of Sturm~\cite{Stmms1}~\cite{Stmms2} and Lott-Villani~\cite{LV}, the notion of lower Ricci curvature bounds can be generalized to the setting of metric and measure spaces. 
\par A Polish metric measure space that satisfies the Lott-Sturm-Villani's $CD(K,N)$ curvature-dimension conditions is called a $CD(K,N)$ space. One important aspect of these spaces is that they support both local and global Poincar\'e inequalities (for a sharp global Poincar\'e inequality and spectral gap on $CD(K,N)$ metric measure spaces, see~\cite[Theorem 5.34]{LV2}). For metric measure spaces that satisfy certain infinitesimal regularity properties, the $CD(K,N)$ curvature-dimension bounds coincide with the Bakry-\'Emery curvature-dimension bounds (or $BE(K,N)$ for short), see~\cite{EKS}. Also, there is a close relation between the lower Ricci bound of $X$ and the lower Ricci curvature bound(s) of the cone(s) over $X$, when $X$ is a Riemannian manifold with $\Ric \ge (n-1)K$ or more generally an $RCD(K,N)$ space. In particular a Riemannian manifold, $X$, satisfies $\Ric \ge 1$ if and only if the Riemannian cone over $X$ satisfies $\Ric \ge 0 $. In the setting of $RCD(K,N)$ metric measure spaces the relation between the weak Ricci curvature bound of $X$ and that of the cone(s) over $X$ has been explored in~\cite{Ketterer}. 
\par There are some disparities between the discrete Laplacian on graphs and the Laplacian on manifolds (or on some more general non-smooth continuous metric measure spaces). Despite these disparities, studying Bakry-\'Emery type curvature-dimension conditions for the discrete Laplacian has proven fruitful in the sense that in the discrete setting graphs with lower Ricci curvature bounds satisfy some properties that are similar to the ones satisfied by manifolds with lower Ricci curvature bounds, see~\cite{LY},~\cite{LLY}, ~\cite{CLY}  and ~\cite{KGPP}.
\par In this paper we acquire partial results relating the curvature of a graph to the curvature of the cone over over its vertices. In general our paper does not admit a clean cut relation between the lower Bakry-\'Emery Ricci curvature bound of the base graph and that of the cone over the graph. This is mainly due to the fact that in the discrete setting the distance between any two vertices in a cone is at most two and thus the operator $\Gamma_2$ at any point $x$ (a key ingredient in the definition of curvature-dimension bounds) will depend on the entire graph. So the curvature bound at the cone point over the vertex set of a graph will store the curvature information of the entire graph, see~\ref{thm:main-1}. 
\par This article is primarily concerned with the properties of the underlying graph $G$ that can be extracted when the cone over $G$ satisfies the $CD(K,N)$ curvature-dimension conditions at the cone point (a property which will be called the conical curvature-dimension, or $CCD(K,N)$ condition). Our main results are a global Poincar\'e inequality and the spectral gap estimates that follow. 
\begin{defn}[$CCD(K,N)$ Curvature-Dimension Conditions]\label{defn:CCD}
Let $G=(V,E)$ be a finite, connected, undirected, loop-edge free graph and consider the cone over the vertex set of $G$. $G$ is said to satisfy the  \emph{conical curvature-dimension condition}, $CCD(K,N)$ for $K \in \mathbb{R}$ and $N \in (1,\infty]$, if the cone over $G$ satisfies the $CD(K,N)$ curvature-dimension conditions at the vertex $p$, namely if
\small
\be
\Gamma^c_2(f)(p)\geq\frac{\bigl(\Delta^c f \bigr)^2(p)}{N}+K\Gamma^c_1(f)(p), \label{eq:ccd}
\ee
\normalsize
holds for any function $f$ defined on the cone and $\Delta^c$, $\Gamma_1^c$ and $\Gamma^c_2$ are the usual $\Delta$, $\Gamma_1$ and $\Gamma_2$ operators (see (\ref{lap}), (\ref{gamma1}) and (\ref{gamma2})) except on the cone $C(G)$ over $G$. We note that the second term in (\ref{eq:ccd}) is understood to be zero when $N=\infty$.
\end{defn}
Now we can state our main theorems and corollaries: 
\begin{theorem}[$CCD(K,N)$ implies global Poincar\'e Inequality]\label{thm:main-1}
	
If a graph, $G$, satisfies $CCD(K,N)$ curvature-dimension condition, then for any function $f$ on $G$ one has
	\small
	\be
	\sum_{y\in V}\Gamma_1(f)(y)\geq\frac{2-N}{2N}\biggl ( \sum_{y\in V}f(y)\biggr )^2+\frac{2K+\lvert V\rvert-3}{4}\sum_{y\in V}f^2(y). \notag
	\ee
\normalsize	
For functions $f$ with $avg(f)=0$, this reduces to the following global Poincar\'e inequality,
	\small
	\be
	\lVert f\rVert_2\leq\sqrt{\frac{2}{2K+\lvert V\rvert-3}}\lVert\nabla f\rVert_2, \notag
	\ee
	\normalsize
    where $\lVert\nabla f\rVert_2$ is understood in the graph setting to be $2\cdot\sum_{y\in V}\Gamma_1(f)(y)$.
\end{theorem}
\begin{corollary}\label{cor:main-1}
	If $G$ satisfies $CCD(K,N)$ condition, then 
	\small
	\be
	\lambda_1(G) \ge K+\frac{\lvert V\rvert-3}{2}. \notag
	\ee
	\normalsize
\end{corollary}
\begin{theorem}\label{thm:main-2}
For any graph, $G$, and a given $N>1$, the \emph{conical curvature} cannot exceed the following number:
\small	
	\be
	K^c_{max} = \frac{\lvert V\rvert}{2}+\frac{3}{2}-2\frac{\lvert V\rvert}{N}. \notag
	\ee
	\normalsize
\end{theorem}
\begin{theorem}[Curvature Maximizers]\label{thm:main-3}
	Suppose $G$ satisfies $CCD(K^c_{max},N)$. Then any function, $f$, realizes $K^c_{max}$ if and only if $f$ is either constant or $f-\operatorname{avg}(f)$ is an eigenfunction corresponding to $\lambda_1(G)=\frac{N-2}{4N}\lvert V\rvert$. Furthermore, when $G$ is a complete graph, $f$ must be constant (harmonic).
\end{theorem}
\begin{corollary}[Ricci Curvature of Complete Graphs]
Suppose $G$ is the complete graph on $n$ vertices, then the $CD(K,N)$ property coincides with the $CCD(K^c,N^c)$ condition on the complete subgraph with $n-1$ vertices and the curvature of $G$ is $\frac{n}{2}+1-2\frac{(n-1)}{N}$. Furthermore any function that realizes this curvature bound is constant (harmonic).
\end{corollary}
\begin{remark}
	When $N=\infty$, our bound $K^c_{max} = 1 + \frac{n}{2} $ coincides with the maximum Ricci curvature of complete graphs as found in~\cite{KGPP}. 
\end{remark}
The following theorem illustrates an applications of our $\Gamma$-calculus on cones:
\begin{theorem}\label{thm:main-5}
    Suppose $G$ satisfies $CD(K,\infty)$ for $K \le \frac{1}{2}$ then the subgraph $G\subset C(G)$ satisfies $CD(K+\frac{1}{2},\infty)$.
\end{theorem}
\subsection*{Acknowledgments}
The authors would like to thank Professor J\'ozef Dodziuk for his interest and encouragement and the Mathematics Department at the CUNY Graduate Center for providing the first author access to their facilities. The authors are also very thankful for the interest and insight of Jorge Basilio. 
\section{Preliminaries}
Let $G=(V,E)$ be an undirected, unweighted, connected, locally finite graph without any loop-edges. Let $f:V\to\mathbb{R}$ and consider the space of square-summable functions on the vertex set. 
 \par The graph Laplacian is given by 
\small
\be
\Delta f(x)=\sum_{y\sim x} \bigl ( f(y)-f(x)  \bigr), \label{lap}
\ee
\normalsize
where $y\sim x$ means that $(y,x)\in E$. Also, note that the graph Laplacian is a real valued, self-adjoint linear operator (for a thorough treatment of the graph Laplacian see~\cite{Dodziuk}).
\par Let $F\subset V$, then the boundary of $F$ is 
\small
\be
	\partial F:=\{(x,y) \in E \mid\  x \in F \text{ and }  y \in V \setminus F   \}. \notag
\ee
\normalsize
The isoperimetric constant (or Cheeger's constant) is then defined as
\small
\be
h(G):=\inf\left\{\dfrac{\lvert\partial F\rvert}{\min\{\lvert F\rvert, \lvert V\setminus F\rvert\}}:\ 0<\lvert F\rvert<\infty\right\}. \notag
\ee
\normalsize
\par A well known generalization of Cheeger's and Buser's results for Riemannian manifolds is the following theorem due to Dodziuk~\cite{Dodziuk} and Alon-Milman~\cite{Alon-Milman}.
\begin{theorem}[\cite{Dodziuk}, \cite{Alon-Milman}]
Let $G=(V,E)$ be a finite, connected, edge-loop free graph. Let $d_{max}=\sup_{v\in V}\{\deg(v)\}$ and let $\lambda_1$ be the first non-trivial eigenvalue of $\Delta$, then
\small
\be
\frac{\lambda_1}{2}\leq h(G)\leq\sqrt{2d_{max}\lambda_1}. \label{eq:DAM}
\ee
\normalsize
\end{theorem}
The $\Gamma$ operators of Bakry-\'Emery associated to the graph Laplacian, $\Delta$, are:
\small
\begin{eqnarray}
\Gamma_1(f,g)(x) &=&\frac{1}{2} \biggl [\Delta(fg)(x)- g(x) \Delta f(x)-f(x)\Delta g(x) \biggr],\label{gamma1}\\
\Gamma_2(f,g)(x) &=&\frac{1}{2} \biggl[ \Delta \Gamma_1(f,g) (x)-\Gamma_1 \bigl ( \Delta f , g \bigr)(x)-\Gamma_1\bigl( f, \Delta g \bigr)(x) \biggr].\label{gamma2}
\end{eqnarray}
\normalsize
It is straightforward to check that 
\small
\be
\Gamma_1(f,g)(x) = \frac{1}{2}\sum_{y\sim x}(f(y)-f(x))(g(y)-g(x)) =: \frac{1}{2} \langle \nabla f, \nabla g \rangle. \label{eq:gradfg}
\ee
\normalsize
Throughout these notes $\Gamma_1 (f) := \Gamma_1(f,f)$ and similarly $\Gamma_2 (f) := \Gamma_2(f,f)$. Thus $\Gamma_1(f)(x)=\frac{1}{2}\lvert\nabla f(x)\rvert^2$ and one can verify the following useful divergence-type identity:
\small
\be\label{eq:divergence}
\frac{1}{2}\lVert\nabla f\rVert_2^2=\sum_{y\in V}\Gamma_1(f)(y)=-\sum_{y\in V}f(y)\Delta f(y).
\ee
\normalsize
\begin{definition}[Bakry-\'Emery Curvature-Dimension Condition]\label{defn:cdkn}
Suppose $K \in \mathbb{R}$ and $N\in(1,\infty]$. We say that a graph $G=(V,E)$ satisfies the curvature-dimension conditions, $CD(K,N)$, if for every $x\in V$ and every $f\in\ell^2(V)$,
\small
 \be
  \Gamma_2(f)(x)\geq\frac{(\Delta f)^2(x)}{N}+K\Gamma_1(f)(x).
 \ee
 \normalsize
 Note when $N=\infty$, the second term in the inequality above is understood to be 0.
\end{definition}
\begin{defn}[Uniform and Pointwise Ricci Curvatures]
We define the dimensional (respectively, dimensionless) Ricci curvature of the graph $G$, $\Ric_{N}(G)$ (respectively, $\Ric_{\infty}(G)$), by
\small
\be
\Ric_N(G) := \sup \left\{    K \ : \ \text{$G$ satisfies $CD(K,N)$}      \right\} \notag
\ee
\normalsize
and 
\small
\be
	\Ric_{\infty}(G) := \sup \left\{    K \ : \ \text{$G$ satisfies $CD(K,\infty)$}      \right\}. \notag
\ee
\normalsize
Similarly, we define the pointwise curvatures by
\small
\be
\Ric_N(y):=\sup\{K:\Gamma_2(f)(y)\geq\frac{1}{N}(\Delta f)^2(y)+K\Gamma_1(f)(y) ,\ \forall f \} \notag
\ee	
\normalsize
and
\small
\be
\Ric_{\infty}(y):=\sup\{K:\Gamma_2(f)(y)\geq K\Gamma_1(f)(y) ,\ \forall f \}. \notag
\ee	
\normalsize
\end{defn}
\begin{definition}[Conical Ricci Curvatures]
We define the conical Ricci curvature by 
\small
\be
CRic_N(G):=\sup\{K: \text{$G$ satisfies $CCD(K,N)$ as in (\ref{eq:ccd})}\} \notag
\ee	
\normalsize
and
\small
\be
CRic_{\infty}(G):=\sup\{K: \text{$G$ satisfies $CCD(K,\infty)$ as in (\ref{eq:ccd})}\}.	\notag
\ee	
\normalsize
\end{definition}
We close this section by recalling that the first non-zero eigenvalue of the Laplacian may be computed via the Rayleigh quotient:
\small
\be\label{eq:Rayleigh}
	\lambda_1 = \inf \left\{ \frac{\| \nabla f \|^2 }{\| f \|^2} :\ \operatorname{avg}(f) = 0  \right\}.
\ee
\normalsize
\section{Cones over Graphs and Their $\Gamma - $ Calculus}
The complete cone, $C(G)$, over a finite graph $G$ is constructed by taking the graph Cartesian product of $G$ and $H$, $G \Box H$, where $H=(\{q,p\},\{(q,p)\})$ is the complete graph on two vertices $q$ and $p$, and then identifying all the vertices whose second component is $p$. In this paper $p$ refers to the cone point of $C(G)$. 
\par More generally for a subset, $X \subset V(G)$, the partial cone, $C\left( X , G \right)$, is a subgraph of $C(G)$ containing  $G$ and all edges $(x,p)$, $x \in X$. For brevity we will use a superscript $c$ to denote any operation that is taking place in a partial cone over $G$. Notice that any vertex $v\in V(G)$ can be thought of as the cone point over the 1-sphere based at $v$, i.e. $S_v^1:=\{y\in V\ \mid \ d_G(y,v)=1\}=X$ in the above construction. In this way partial cones can be useful in studying cliques.
\par The first subsection is devoted to proving a few lemmas that calculate the $\Delta$ and $\Gamma$ operators of a partial cone in terms of the similar operators on the base graph. The last subsection is devoted to an immediate result.
\subsection{$\Gamma$-Calculus on a Cone}
\par Since $\Delta$ and $\Gamma$ operators agree for functions that differ by a constant we may assume, without loss of generality, that $f(p)=0$. 
Denote by $S^n_p$ and $B^n_p$ the metric spheres and balls (resp.) with radius $n$ and center $p$ in the cone. For any subset $B \subset V$, the notation $v \in B\sim x$ means $v \in B$ and $v \sim x$. 
\begin{remark}
 Note that $\Delta$ and $\Gamma_1$ only depend on vertices that are at most one away. Thus, $\Delta^c f(x)=\Delta f(x)$ and $\Gamma^c_1(f)(x)=\Gamma_1(f)(x)$ when $x\nsim p$.
\end{remark}
\begin{lemma}\label{lem:cone-Delta}
Let $f$ be a function on the cone with $f(p) =0$ then,	
\small
\be	
\Delta^cf(x)=\begin{cases}
	\Delta f(x)-f(x); & x\sim p\\
	\sum_{y\in S^1_p}f(y); & x=p\\
\end{cases} \notag	
\ee	
\normalsize
\end{lemma}
\begin{proof}

\begin{enumerate}
\item 	If $x\sim p$, then 
	\small
	\be
	\Delta^c f (x)=\sum_{y\in C\sim x} \bigl( f(y)-f(x) \bigr) = \sum_{y\in V\sim x} \bigl( f(y)-f(x) \bigr)+ \bigl( f(p)-f(x) \bigr) = \Delta f(x)-f(x). \notag
	\ee
	\normalsize
\item If $x=p$, then
	\small
	\be
	\Delta^c f (p)= \sum_{y\in C\sim p} \bigl( f(y)-f(p) \bigl) = \sum_{y\in S^1_p} f(y). \notag
	\ee
	\normalsize 
	\end{enumerate}
\end{proof}
\begin{lemma}\label{lem:cone-Gamma1}
Let $f$ be a function on the cone with $f(p) =0$ then,		
\small
\be
\Gamma_1^c(f)(x)=\begin{cases}
	\Gamma_1(f)(x)+\frac{1}{2}f^2(x); & x\sim p\\
	\frac{1}{2}\sum_{y\in S^1_p}f^2(y); & x=p\\
	\end{cases} \notag
\ee
\normalsize
\end{lemma}
\begin{proof}
\begin{enumerate}
\item  If $x\sim p$, then using (\ref{eq:gradfg})
\small
	\be
	\Gamma_1^c(f)(x)=\frac{1}{2}\sum_{y\in C\sim x} \bigl( f(y)-f(x) \bigr)^2 = \frac{1}{2}\sum_{y\in V\sim x} \bigl( f(y)-f(x) \bigl)^2 + \frac{1}{2} \bigl( f(p)-f(x) \bigl)^2= \Gamma_1(f)(x)+\frac{1}{2}f^2(x). \notag
	\ee
\normalsize
\item If $x=p$, then using (\ref{eq:gradfg})
	\small
	\be
	\Gamma_1^c(f)(p)=\frac{1}{2}\sum_{y\in V} \bigl( f(y)-f(p) \bigr)^2 =\frac{1}{2}\sum_{y\in S^1_p} f^2(y). \notag
	\ee
	\normalsize 
\end{enumerate}	
\end{proof}
In the next few lemmas we calculate the constituent parts that appear in the definition of $\Gamma^c_2$.
\begin{remark}
Note that $\Gamma^c_2$ depends on vertices at most two away. Thus $\Gamma^c_2$ coincides with $\Gamma_2$ when $x\in V \setminus B_p^2$.
\end{remark}
\begin{lemma}
Let $f$ be a function defined on the cone, and suppose $f(p) = 0$, then	
\small
\be
\Gamma_1^c \bigl( f,\Delta^cf \bigr)(x)=\begin{cases}
	\Gamma_1 \bigl( f,\Delta f \bigr)(x)-\frac{1}{2}\sum_{y\in S^1_p\sim x}f(y) \bigl( f(y)-f(x) \bigr); & x\in S^2_p\\
	\Gamma_1 \bigl( f,\Delta f \bigr)(x)-\frac{1}{2}\sum_{y\in S^1_p\sim x} \bigl( f(y)-f(x) \bigr)^2\\ 
	    \qquad\qquad\qquad+\frac{1}{2}f(x)\sum_{y\in S^2_p\sim x} \bigl( f(y)-f(x) \bigr)\\ 
	    \qquad\qquad\qquad-\frac{1}{2}f(x)\sum_{y\in S^1_p}f(y)+\frac{1}{2}f(x)\Delta f(x)-\frac{1}{2}f^2(x) & x\sim p\\
	\frac{1}{2}\sum_{y\in S^1_p}f(y)\Delta f(y)-\frac{1}{2}\sum_{y\in S^1_p}f^2(y)-\frac{1}{2}(\sum_{y\in S^1_p}f(y))^2 & x=p
	\end{cases}\notag
\ee
\normalsize
\end{lemma}
\begin{proof}
\begin{enumerate}	
\item 	If $x\in S^2_p$, then using (\ref{eq:gradfg})
\small	
	\begin{eqnarray*}
	\Gamma_1^c \bigl( f, \Delta^c f \bigr)(x) &=& \frac{1}{2}\sum_{y\in C\sim x} \bigl( f(y)-f(x) \bigr) \bigl( \Delta^c f(y)-\Delta^c f(x) \bigr)\\
	&=&\frac{1}{2}\sum_{y\in V \setminus S_p^1 \sim x} \bigl( f(y)-f(x) \bigr) \bigl( \Delta^c f(y)-\Delta^c f(x) \bigr) \\ 
	&& + \; \frac{1}{2}\sum_{y\in S^1_p\sim x}\bigl( f(y)-f(x) \bigr) \bigl( \Delta^c f(y)-\Delta^c f(x) \bigr)\\
	&=&\frac{1}{2}\sum_{y\in V \setminus S_p^1 \sim x}\bigl( f(y)-f(x) \bigr) \bigl( \Delta f(y)-\Delta f(x) \bigr) \\ 
	&&+ \; \frac{1}{2}\sum_{y\in S^1_p\sim x} \bigl( f(y)-f(x) \bigr) \bigl( \Delta f(y)-f(y)-\Delta f(x) \bigr)\\
	&=&\frac{1}{2}\sum_{y\in V\sim x} \bigl( f(y)-f(x) \bigr) \bigl( \Delta f(y)-\Delta f(x) \bigr)-\frac{1}{2}\sum_{y\in S^1_p\sim x}f(y)(f(y)-f(x))\\
	&=&\Gamma_1 \bigl( f,\Delta f \bigr)(x)-\frac{1}{2}\sum_{y\in S^1_p\sim x}f(y) \bigl( f(y)-f(x) \bigr) .
	\end{eqnarray*}
	\normalsize
\item 	If $x\sim p$, then using (\ref{eq:gradfg})
	\small
	\begin{eqnarray*}
	\Gamma_1^c \bigl( f,\Delta^c f \bigr) (x) &=&\frac{1}{2}\sum_{y\in C\sim x} \bigl( f(y)-f(x) \bigr) \bigl( \Delta^c f(y)-\Delta^c f(x) \bigr) \\
	&=&\frac{1}{2}\sum_{y\in S^2_p\sim x}\bigl ( f(y)-f(x) \bigr) \bigl( \Delta f(y)-\Delta f(x)+f(x) \bigr)\\
	&&+ \; \frac{1}{2}\sum_{y\in S^1_p\sim x} \bigl( f(y)-f(x) \bigr) \bigl( \Delta f(y)-f(y)-\Delta f(x)+f(x) \bigr)\\
	&&+ \; \frac{1}{2} \bigl( f(p)-f(x) \bigr) \bigl( \sum_{y\in S^1_p}f(y)-\Delta f(x)+f(x) \bigr)\\
	&=&\frac{1}{2}\sum_{y\in S^2_p\sim x} \bigl( f(y)-f(x) \bigr) \bigl( \Delta f(y)-\Delta f(x) \bigr) +\frac{1}{2}f(x)\sum_{y\in S^2_p\sim x} \bigl( f(y)-f(x) \bigr)\\
	&&+ \; \frac{1}{2}\sum_{y\in S^1_p\sim x}\bigl( f(y)-f(x) \bigr) \bigl( \Delta f(y)-\Delta f(x) \bigr) -\frac{1}{2}\sum_{y\in S^1_p\sim x} \bigl( f(y)-f(x) \bigr)^2\\
	&&- \; \frac{1}{2}f(x)\sum_{y\in S^1_p}f(y)+\frac{1}{2}f(x)\Delta f(x)-\frac{1}{2}f^2(x)\\
	&=&\frac{1}{2}\sum_{y\in V\sim x} \bigl( f(y)-f(x) \bigr) \bigl( \Delta f(y)-\Delta f(x) \bigr)-\frac{1}{2}\sum_{y\in S^1_p\sim x}\bigl( f(y)-f(x) \bigr)^2\\ &&+ \; \frac{1}{2}f(x)\sum_{y\in S^2_p\sim x}\bigl( f(y)-f(x) \bigr)-\frac{1}{2}f(x)\sum_{y\in S^1_p}f(y)+\frac{1}{2}f(x)\Delta f(x)-\frac{1}{2}f^2(x)\\
	&=&\Gamma_1 \bigl( f,\Delta f \bigr) (x)-\frac{1}{2}\sum_{y\in S^1_p\sim x} \bigl( f(y)-f(x) \bigr)^2 +\frac{1}{2}f(x)\sum_{y\in S^2_p\sim x} \bigl( f(y)-f(x) \bigr) \\ &&- \; \frac{1}{2}f(x)\sum_{y\in S^1_p}f(y)+\frac{1}{2}f(x)\Delta f(x)-\frac{1}{2}f^2(x).
	\end{eqnarray*}
\normalsize
\item  If $x=p$, then using (\ref{eq:gradfg})
	\small
	\begin{eqnarray*}
	\Gamma_1^c \bigl( f,\Delta^c f \bigr) (p) &=&\frac{1}{2}\sum_{y\in C\sim p} \bigl( f(y)-f(p) \bigr) \bigl( \Delta^cf(y)-\Delta^cf(p) \bigr)\\
	&=&\frac{1}{2}\sum_{y\in S^1_p}\bigl( f(y)-f(p) \bigr) \bigl( \Delta f(y)-f(y)-\sum_{z\in S^1_p}f(z) \bigr)\\
	&=&\frac{1}{2}\sum_{y\in S^1_p}\biggl [f(y)\Delta f(y)-f^2(y)-f(y)\sum_{z\in S^1_p}f(z) \biggr ]\\
	&=&\frac{1}{2}\sum_{y\in S^1_p}f(y)\Delta f(y)-\frac{1}{2}\sum_{y\in S^1_p}f^2(y)-\frac{1}{2} \bigl( \sum_{y\in S^1_p}f(y) \bigl)^2.
	\end{eqnarray*}
\normalsize
\end{enumerate}
\end{proof}
\begin{lemma}
Let $f$ be a function defined on the cone, and suppose $f(p) = 0$, then
\small
\be
\Delta^c\Gamma_1^c(f)(x)=\begin{cases}
	\Delta\Gamma_1(f)(x)+\frac{1}{2}\sum_{y\in S^1_p\sim x}f^2(y); & x\in S_p^2\\
	\Delta\Gamma_1(f)(x)-\Gamma_1(f)(x)\\
	\qquad\qquad\quad+\frac{1}{2}\biggl[\sum_{y\in S^1_p\sim x}f^2(y)+\sum_{y\in S^1_p}f^2(y)\biggr]\\
	\qquad\qquad\quad-\frac{1}{2}\deg(x)f^2(x) & x\sim p\\
	\sum_{y\in S^1_p}\Gamma_1(f)(y)-\frac{\lvert S^1_p\rvert-1}{2}\sum_{y\in S^1_p}f^2(y); & x=p
	\end{cases} \notag
\ee
\normalsize
\end{lemma}
\begin{proof}
\begin{enumerate}	
\item If $x\in S^2_p$, then
\small
	\begin{eqnarray*}
	\Delta^c\Gamma_1^c(f)(x) &=&\sum_{y\in C\sim x} \biggl[ \Gamma_1^c(f)(y)-\Gamma_1^c(f)(x) \biggr]\\
	&=&\sum_{y\in V \setminus S_p^1 \sim x} \biggl[ \Gamma_1^c(f)(y)-\Gamma_1^c(f)(x) \biggr]+\sum_{y\in S^1_p\sim x} \biggl[ \Gamma_1^c(f)(y)-\Gamma_1^c(f)(x) \biggr]\\
	&=&\sum_{y\in V \setminus S_p^1 \sim x} \biggl[ \Gamma_1(f)(y)-\Gamma_1(f)(x) \biggr] +\sum_{y\in S^1_p\sim x} \biggl[ \Gamma_1(f)(y)+\frac{1}{2}f^2(y)-\Gamma_1(f)(x) \biggr]\\
	&=&\sum_{y\in V\sim x} \biggl[ \Gamma_1(f)(y)-\Gamma_1(f)(x) \biggr] + \frac{1}{2}\sum_{y\in S^1_p\sim x}f^2(y)\\
	&=&\Delta\Gamma_1(f)(x)+\frac{1}{2}\sum_{y\in S^1_p\sim x}f^2(y).
	\end{eqnarray*}
	\normalsize
\item 	 If $x\sim p$, then
\small
	\begin{eqnarray*}
	\Delta^c\Gamma_1^c(f)(x) &=&\sum_{y\in C\sim x} \biggl[ \Gamma_1^c(f)(y)-\Gamma_1^c(f)(x) \biggr] \\
	&=&\sum_{y\in S^2_p\sim x} \biggl[ \Gamma_1(f)(y)-\Gamma_1(f)(x)-\frac{1}{2}f^2(x) \biggr]
	\\
	&&+ \; \sum_{y\in S^1_p\sim x} \biggl[ \Gamma_1(f)(y)-\Gamma_1(f)(x)+\frac{1}{2} \bigl( f^2(y)-f^2(x) \bigr) \biggr] + \Gamma_1^c(f)(p)-\Gamma_1(f)(x)-\frac{1}{2}f^2(x)\\
	&=&\sum_{y\in S^2_p\sim x} \biggl[ \Gamma_1(f)(y)-\Gamma_1(f)(x) \biggr]-\frac{1}{2}\deg(x) f^2(x)
	+\sum_{y\in S^1_p\sim x} \biggl[ \Gamma_1(f)(y)-\Gamma_1(f)(x) \biggr]\\
	&&+ \; \frac{1}{2}\sum_{y\in S^1_p\sim x} f^2(y)+\frac{1}{2}\sum_{y\in S^1_p}f^2(y)-\Gamma_1(f)(x)\\
	&=&\sum_{y\in V\sim x} \biggl[ \Gamma_1(f)(y)-\Gamma_1(f)(x) \biggr] - \Gamma_1(f)(x)\\
	&&+ \; \frac{1}{2}\sum_{y\in S^1_p\sim x} f^2(y) + \frac{1}{2}\sum_{y\in S^1_p}f^2(y)-\frac{1}{2} \deg(x)f^2(x)\\
	&=&\Delta\Gamma_1(f)(x)-\Gamma_1(f)(x)+\frac{1}{2}\biggl[\sum_{y\in S^1_p\sim x}f^2(y)+\sum_{y\in S^1_p}f^2(y)\biggr]
	-\frac{1}{2}\deg(x)f^2(x).
	\end{eqnarray*}
	\normalsize
\item	 If $x=p$, then
\small
	\begin{eqnarray*}
	\Delta^c \bigl( \Gamma_1^c(f) \bigr) (p) &=&\sum_{y\in S^1_p} \biggl[ \Gamma_1^c(f)(y)-\Gamma_1^c(f)(p) \biggr]\\
	&=&\sum_{y\in S^1_p} \biggl[ \Gamma_1(f)(y)+\frac{1}{2}f^2(y)-\frac{1}{2}\sum_{y\in S^1_p}f^2(y) \biggr]\\
	&=&\sum_{y\in S^1_p}\Gamma_1(f)(y)+\frac{1}{2}\sum_{y\in S^1_p}f^2(y)-\frac{\lvert S^1_p\rvert}{2}\sum_{y\in S^1_p}f^2(y)\\
	&=&\sum_{y\in S^1_p}\Gamma_1(f)(y)-\frac{(\lvert S^1_p\rvert-1)}{2}\sum_{y\in S^1_p}f^2(y).
	\end{eqnarray*}
\normalsize
\end{enumerate}
\end{proof}
\begin{lemma}\label{lem:cone-Gamma2}
Let $f$ be a function defined on the cone, and suppose $f(p) = 0$, then
\small
\be
\Gamma_2^c(f)(x)=\begin{cases}
	\Gamma_2(f)(x)+\frac{3}{4}\sum_{y\in S^1_p\sim x}f^2(y)-\frac{1}{2}f(x)\sum_{y\in S^1_p\sim x}f(y); & x\in S_p^2\\
	\Gamma_2(f)(x)- \frac{1}{2}  \Gamma_1(f)(x)+\frac{1}{2}\sum_{y\in S^1_p\sim x} \bigl( f(y)-f(x) \bigr)^2\\
	    \qquad\qquad+\frac{1}{4}\biggl[\sum_{y\in S^1_p\sim x}f^2(y)-\deg(x)f^2(x)\biggr]-\frac{1}{2}f(x)\Delta f(x)\\
	    \qquad\qquad-\frac{1}{2}f(x)\sum_{y\in S^2_p\sim x}(f(y)-f(x))+\frac{1}{4}\sum_{y\in S^1_p}f^2(y)+\frac{1}{2} f^2(x)  ; & x\sim p\\
	\frac{1}{2}\sum_{y\in S^1_p}\Gamma_1(f)(y)-\frac{1}{2}\sum_{y\in S^1_p}f(y)\Delta f(y)\\
	    \qquad\qquad -\frac{\vert S^1_p\rvert-3}{4}\sum_{y\in S^1_p}f^2(y)+\frac{1}{2} \bigl( \sum_{y\in S^1_p}f(y) \bigr )^2; & x=p
	\end{cases} \notag
	\ee
	\normalsize
\end{lemma}
\begin{proof}
\begin{enumerate}
\item  If $x\in S^2_p$, then
\small	
	\begin{eqnarray*}
	\Gamma_2^c(f)(x) &=&\frac{1}{2}\Delta^c\Gamma_1^c(f,f)(x)-\Gamma_1^c \bigl( f,\Delta^c f \bigr)(x)\\
	&=&\frac{1}{2}\Delta\Gamma_1(f)(x)+\frac{1}{4}\sum_{y\in S^1_p\sim x}f^2(y)-\Gamma_1\bigl( f,\Delta f \bigr)(x)+\frac{1}{2}\sum_{y\in S^1_p\sim x}f(y) \bigl( f(y)-f(x) \bigr)\\
	&=&\Gamma_2(f)(x)+\frac{3}{4}\sum_{y\in S^1_p\sim x}f^2(y)-\frac{1}{2}f(x)\sum_{y\in S^1_p\sim x}f(y).
	\end{eqnarray*}
	\normalsize
\item If $x\sim p$, then
 \small
	\begin{eqnarray*}
		\Gamma_2^c(f)(x) &=&\frac{1}{2}\Delta^c\Gamma_1^c(f)(x)-\Gamma_1^c\bigl( f,\Delta^c f \bigr)(x)\\
	&=&\frac{1}{2}\Delta\Gamma_1(f)(x)-\frac{1}{2}\Gamma_1(f)(x)+\frac{1}{4}\sum_{y\in S^1_p\sim x} f^2(y)+\frac{1}{4}\sum_{y\in S^1_p}f^2(y)-\frac{1}{4}\deg(x)f^2(x)\\
	&&- \; \Gamma_1\bigl( f,\Delta f \bigr)(x)+\frac{1}{2}\sum_{y\in S^1_p\sim x} \bigl( f(y)-f(x) \bigr)^2  -\frac{1}{2}f(x)\sum_{y\in S^2_p \sim x} \bigl( f(y)-f(x) \bigr)\\
	&&+ \; \frac{1}{2}f(x)\sum_{y\in S^1_p}f(y) -\frac{1}{2}f(x)\Delta f(x)+\frac{1}{2}f^2(x).\\
	&=&  \biggl[ \frac{1}{2}\Delta\Gamma_1(f)(x)-\Gamma_1 \bigl( f,\Delta f \bigr)(x)  \biggr] - \frac{1}{2}  \Gamma_1(f)(x)+\frac{1}{2}\sum_{y\in S^1_p\sim x} \bigl( f(y)-f(x) \bigr)^2  \\
	&&+ \; \frac{1}{4}\biggl[\sum_{y\in S^1_p\sim x}f^2(y)-\deg(x)f^2(x)\biggr]+\frac{1}{4}\sum_{y\in S^1_p}f^2(y)-\frac{1}{2}f(x)\sum_{y\in S^2_p\sim x}(f(y)-f(x))\\
	&&- \; \frac{1}{2}f(x)\Delta f(x)+\frac{1}{2} f^2(x)\\
	&=&\Gamma_2(f)(x)- \frac{1}{2}  \Gamma_1(f)(x)+\frac{1}{2}\sum_{y\in S^1_p\sim x} \bigl( f(y)-f(x) \bigr)^2-\frac{1}{2}f(x)\sum_{y\in S^2_p\sim x}(f(y)-f(x))  \\
	&&+ \; \frac{1}{4}\biggl[\sum_{y\in S^1_p\sim x}f^2(y)-\deg(x)f^2(x)\biggr]-\frac{1}{2}f(x)\Delta f(x)+\frac{1}{4}\sum_{y\in S^1_p}f^2(y)+\frac{1}{2} f^2(x)
	\end{eqnarray*}
	\normalsize
\item  If $x=p$, then
\small
	\begin{eqnarray*}
	\Gamma_2^c(f)(p)&=&\frac{1}{2}\Delta^c\Gamma_1^c(f)(p)-\Gamma_1^c \bigl( f,\Delta^c f \bigr)(p)\\
	&=&\frac{1}{2}\sum_{y\in S^1_p}\Gamma_1(f)(y)-\frac{\lvert S^1_p\rvert-1}{4}\sum_{y\in S^1_p}f^2(y)-\frac{1}{2}\sum_{y\in S^1_p}f(y)\Delta f(y)+\frac{1}{2}\sum_{y\in S^1_p}f^2(y)+\frac{1}{2}  \bigl(\sum_{y\in S^1_p}f(y) \bigr)^2\\
	&=&\frac{1}{2}\sum_{y\in S^1_p}\Gamma_1(f)(y)-\frac{1}{2}\sum_{y\in S^1_p}f(y)\Delta f(y)-\frac{\lvert S^1_p\rvert-3}{4}\sum_{y\in S^1_p}f^2(y)+\frac{1}{2}\bigl(\sum_{y\in S^1_p}f(y)\bigr )^2 . 
	\end{eqnarray*}
	\normalsize
\end{enumerate}	
\end{proof}
\subsection{$\Gamma^c_2$ for $C(G)$}

When $C=(V^c,E^c)$ is the full cone over $V(G)$, then $S^1_p=V$ and so Lemma~\ref{lem:cone-Gamma2} reduces to
\begin{lemma}\label{lem:cone-Gamma3}
\small
\be
\Gamma_2^c(f)(x)=\begin{cases}
	\Gamma_2(f)(x)+\frac{1}{2}\Gamma_1(f)(x)+\frac{1}{4}\sum_{y\in V}f^2(y)+\frac{1}{2}f^2(x); & x\sim p\\
	\sum_{y\in V}\Gamma_1(f)(y)-\frac{\lvert V\rvert-3}{4}\sum_{y\in V}f^2(y)+\frac{1}{2} \bigl( \sum_{y\in V}f(y) \bigl )^2; & x=p\\
	\end{cases} \notag.
	\ee
	\normalsize
\end{lemma}
\begin{proof}
	Since $S^1_p=V$ and $S^2_p=\varnothing$ the first case in Lemma~\ref{lem:cone-Gamma2} disappears. In case 2 notice that when $S^1_p=V$, then $\frac{1}{2}\sum_{y\in S^1_p\sim x} \bigl( f(y)-f(x) \bigr)^2=\Gamma_1(f)(x)$ and $\sum_{y\in S^1_p\sim x} \bigl( f^2(y)-f^2(x) \bigr)= \bigl( \Delta f^2 \bigr)(x)$. Since $\frac{1}{2}\Gamma_1(f)(x)= \frac{1}{4} \bigl( \Delta f^2 \bigr)(x)-\frac{1}{2}f(x)\Delta f(x)$ the case when $x\sim p$ follows. When $x=p$, applying the identity (\ref{eq:divergence}) gives the desired result.
\end{proof}
This leads to the following result regarding the curvature of the cone,  
\begin{customthm}{\ref{thm:main-5}}
    Suppose $G$ satisfies $CD(K,\infty)$ for $K \le \frac{1}{2}$ then the subgraph $G\subset C(G)$ satisfies $CD(K+\frac{1}{2},\infty)$.
\end{customthm}
\begin{proof}[Proof of Theorem~\ref{thm:main-5}]
Suppose $G$ satisfies $CD(K,\infty)$ for $K \le \frac{1}{2}$. Since $G$ satisfies $CD(K,\infty)$ then by lemma~\ref{lem:cone-Gamma3} for $x\sim p$,
    \small
    \be
    \Gamma_2^c(f)(x)\geq (K+1)\Gamma_1(f)(x)+\frac{1}{4}\sum_{y\in V}f^2(y)+\frac{1}{2}f^2(x). \notag
    \ee
    \normalsize
    Therefore,
    \small
    \be
    \Gamma_2^c(f)(x)\geq (K+1)\Gamma^c_1(f)(x)+\frac{1}{4}\sum_{y\in V}f^2(y)-\frac{K}{2}f^2(x). \notag
	\ee
    \normalsize
Since $K \le \frac{1}{2}$, then $\frac{1}{4}\sum_{y\in V}f^2(y)-\frac{K}{2}f^2(x)\geq0$. Hence we may drop both terms from the inequality and $C(G)$ satisfies $CD(K+1,\infty)$ for $x\sim p$.
\end{proof}
\section{$CCD(K,N)$ and Global Poincar\'e Inequality }
\par If the cone $C$ satisfies the $CD(K,N)$ inequality at the vertex, $p$, then by Lemmas~\ref{lem:cone-Gamma1} and~\ref{lem:cone-Delta}, we get
    \small
    \be\label{eq:cd-vertex-1}
    \Gamma_2^c(f)(p)\geq\frac{1}{N}\bigl (\sum_{y\in V}f(y) \bigr )^2+\frac{K}{2}\sum_{y\in V}f^2(y). 
    \ee
    \normalsize
This leads to the following,
\begin{customthm}{\ref{thm:main-1}}
	
	If a graph, $G$, satisfies $CCD(K,N)$ curvature-dimension condition, then for any function $f$ on $G$ one has
	\small
	\be\label{eq:poincare}
	\sum_{y\in V}\Gamma_1(f)(y)\geq\frac{2-N}{2N}\biggl ( \sum_{y\in V}f(y)\biggr )^2+\frac{2K+\lvert V\rvert-3}{4}\sum_{y\in V}f^2(y).
	\ee
\normalsize	
For functions $f$ with $avg(f)=0$, this reduces to the following global Poincar\'e inequality,
	\small
	\be
	\lVert f\rVert_2\leq\sqrt{\frac{2}{2K+\lvert V\rvert-3}}\lVert\nabla f\rVert_2, \notag
	\ee
	\normalsize
    where $\lVert\nabla f\rVert_2$ is understood in the graph setting to be $2\cdot\sum_{y\in V}\Gamma_1(f)(y)$.
\end{customthm}
\begin{proof}[Proof of Theorem~\ref{thm:main-1}]
	Suppose a graph $G$ satisfies $CCD(K,N)$ condition. By Lemma~\ref{lem:cone-Gamma3},
    \small	
	\be\label{eq:cd-vertex-2}
    \Gamma_2^c(f)(p)=\sum_{y\in V}\Gamma_1(f)(y)-\frac{\lvert V\rvert-3}{4}\sum_{y\in V}f^2(y)+\frac{1}{2}\bigl (\sum_{y\in V}f(y)\bigr )^2, 
    \ee 
    \normalsize
Upon combining (\ref{eq:cd-vertex-2}) and (\ref{eq:cd-vertex-1}), we will arrive at
	\small
	\be
    \sum_{y\in V}\Gamma_1(f)(y)-\frac{\lvert V\rvert-3}{4}\sum_{y\in V}f^2(y)+\frac{1}{2}\bigl (\sum_{y\in V}f(y)\bigr )^2\geq\frac{1}{N}\bigl (\sum_{y\in V}f(y)\bigr )^2+\frac{K}{2}\sum_{y\in V}f^2(y), \notag
    \ee
    \normalsize
which simplifies to
    \small
    \be
    \sum_{y\in V}\Gamma_1(f)(y)\geq\frac{2-N}{2N}\bigl( \sum_{y\in V}f(y)\bigr)^2+\frac{2K+\lvert V\rvert-3}{4}\sum_{y\in V}f^2(y). \notag
    \ee
    \normalsize
For $f$ with $\operatorname{avg}(f) = 0$, the above reduces to
    \small
	\be
	\frac{1}{2}\sum_{y\in V}\lvert\nabla f(y)\rvert^2\geq\frac{2K+\lvert V\rvert-3}{4}\sum_{y\in V}f^2(y).
	\ee
	\normalsize	
By the definition in ~\ref{eq:gradfg} this yields the Poincar\'e inequality,	
    \small
	\be
	\lVert f \rVert_2\leq\sqrt{\frac{2}{2K+\lvert V\rvert-3}}\lVert\nabla f\rVert_2  \text{, when } \operatorname{avg}(f) =0. \notag
	\ee
	\normalsize
\end{proof}
\begin{customthm}{\ref{thm:main-2}}
For any graph, $G$, and a given $N>1$, the \emph{conical curvature} cannot exceed the following number:
\small	
	\be
	K^c_{max} = \frac{\lvert V\rvert}{2}+\frac{3}{2}-2\frac{\lvert V\rvert}{N}. \notag
	\ee
	\normalsize
\end{customthm}
\begin{proof}[Proof of Theorem~\ref{thm:main-2}]
	Suppose a finite graph $G$ satisfies the $CCD(K,N)$ and $f$ is a non-zero harmonic function, then one has $\sum_{y\in V}\Gamma_1(f)(y)=0$  (i.e. $f$ is constant on connected components). Thus,
	\small
	\begin{eqnarray*}
	\frac{2K+\lvert V\rvert-3}{4}\sum_{y\in V}f^2(y) \le \frac{N-2}{2N}\bigl ( \sum_{y\in V}f(y)\bigr)^2. \notag
	\end{eqnarray*}
	\normalsize
By the Cauchy-Schwarz inequality, 
\small
\be
\bigl ( \sum_{y\in V}f(y)\bigr )^2 \le  \lvert V\rvert\cdot\sum_{y\in V}f^2(y), \notag
\ee
\normalsize
which implies
\small
\be
\frac{2K+\lvert V\rvert-3}{4}\sum_{y\in V}f^2(y) \le \frac{N-2}{2N}\lvert V\rvert\cdot\sum_{y\in V}f^2(y) . \notag
\ee
\normalsize
Since, $f$ is not constant zero, 
\small
\be
K \le \frac{\lvert V\rvert}{2}-\frac{2\lvert V\rvert}{N}+\frac{3}{2}. \notag
\ee
\normalsize
\end{proof}
Having established an upper bound for the curvature at the cone point over the vertex set of the graph $G$ we now turn to an investigation of when the maximum curvature value is achieved.
\begin{lemma}
For any finite graph, $G$, the Ricci curvatures $\Ric_{\infty}(G)$, $\Ric_N(G)$, $CRic_{\infty}(G)$ and $CRic_N(G)$ are realized by some functions, i.e. there are functions that achieve the equality in the (corresponding) defining Bakry-\'Emery curvature-dimension inequalities. 
\end{lemma}
\begin{proof}
	We will just present a proof for $\Ric_N(G)$. The proof for other Ricci curvatures are similar. 
\par Since, $\Ric_N(G)$ is the supremum of all possible lower curvature bounds, one can find a sequence, $g_i$ such that for all $v \in V$
\small
\be\label{eq:max-Ric}
	\frac{1}{N} \bigl( \Delta g_i \bigr)^2(v)+ \Ric_N(G)\Gamma_1 \bigl( g_i \bigr)(v) \le \Gamma_2 \bigl( g_i \bigr)(v) <  \frac{1}{N}\bigl( \Delta g_i \bigr)^2(v)+\bigl (\Ric_N(G)+\frac{1}{i}\bigr)\Gamma_1\bigl( g_i \bigr)(v). 
\ee	
\normalsize
All the terms appearing in the above inequality are invariant under rescaling of the $g_i$'s. Hence, without loss of generality, we may assume that  $\operatorname{Range}(g_i) \subset [-1, 1]$, for all $i$. Now since $V(G)$ is finite then by a diagonal argument one can find a subsequence $g_j$ of the $g_i$'s that converge to a function $g$. Taking the limit of~(\ref{eq:max-Ric}) as $j \to \infty$ shows that $g$ achieves $\Ric_N(G)$.
\end{proof}
\section{Functions That Maximize the Conical Curvature}\label{sec:max}
In this section we show that 
\begin{customthm}{\ref{thm:main-3}}
	Suppose $G$ satisfies $CCD(K^c_{max},N)$. Then any function, $f$, realizes $K^c_{max}$ if and only if $f$ is either constant or $f-\operatorname{avg}(f)$ is an eigenfunction corresponding to $\lambda_1(G)=\frac{N-2}{4N}\lvert V\rvert$. Furthermore, when $G$ is a complete graph, $f$ must be constant (harmonic).
\end{customthm}

\begin{proof}
	Suppose $G$ satisfies $CCD(K^c_{max},N)$, then for any $f$
	\small
	\be
	\sum_{y\in V}\Gamma_1(f)(y)\geq\frac{2-N}{2N}\bigl( \sum_{y\in V}f(y)\bigr)^2+\frac{2K^c_{max}+\lvert V\rvert-3}{4}\sum_{y\in V}f^2(y). \notag
	\ee
	\normalsize
	Since $K^c_{max} = \frac{\lvert V\rvert}{2}+\frac{3}{2}-2\frac{\lvert V\rvert}{N}=\frac{N\cdot\lvert V\rvert+3N-4\lvert V\rvert}{2N}$, this simplifies to
\small
	\begin{eqnarray*}
	\sum_{y\in V}\Gamma_1(f)(y) &\geq&\frac{2-N}{2N}\bigl ( \sum_{y\in V}f(y)\bigr )^2+\frac{N\cdot\lvert V\rvert+3N-4\lvert V\rvert+N\cdot\lvert V\rvert-3N}{4N}\sum_{y\in V}f^2(y)\\
	&\geq&\frac{2-N}{2N}\bigl ( \sum_{y\in V}f(y)\bigr )^2+\frac{N\cdot\lvert V\rvert-2\lvert V\rvert}{2N}\sum_{y\in V}f^2(y)\\
	&\geq&\frac{2-N}{2N}\bigl ( \sum_{y\in V}f(y)\bigr )^2+\frac{N-2}{2N}\cdot\lvert V\rvert\sum_{y\in V}f^2(y)\\
	&\geq&\frac{N-2}{2N}\biggl[ \lvert V\rvert\sum_{y\in V}f^2(y)-\bigl ( \sum_{y\in V}f(y)\bigr )^2\biggr ] . 
	\end{eqnarray*}
	\normalsize
	Take $\phi:V\to\mathbb{R}$ to be any variational function on the vertex set of $G$ and let $t \in \R$, then 
\small
\be
\sum_{y\in V}\Gamma_1 \bigl( f+t\phi \bigr)(y)\geq\frac{N-2}{2N}\biggl [ \lvert V\rvert\sum_{y\in V}(f+t\phi)^2(y)-\bigl ( \sum_{y\in V} (f+t\phi)(y)\bigr )^2\biggr ]. \notag
\ee	
\normalsize
Suppose now that $f$ achieves $K^c_{max}$, then for any $\phi$ the above inequality becomes an equality (i.e. $\frac{d}{dt}|_{t=0}$ of both sides must be equal for any variation $\phi$). Hence a straightforward calculation yields the linearized equation,
\small
\be\label{eq:linearization-1}
	\sum_{y\in V}\sum_{z\sim y} \bigl( f(z)-f(y) \bigr) \bigl( \phi(z)-\phi(y) \bigr) = \frac{N-2}{2N}\biggl [\lvert V\rvert\sum_{y\in V}f(y)\phi(y)-\sum_{y\in V}f(y)\sum_{y\in V}\phi(y)\biggr ]. 
\ee
\normalsize
Now fix $r\in V$ and let $\phi(y) = \delta_r(y)$. Notice that $\sum_{z\sim y} \bigl( f(z)-f(y) \bigr) \bigl( \delta_r(z)-\delta_r(y) \bigr)$ is zero except when $y=r$ or $y\sim r$. If $y=r$ the result is $-\sum_{z\sim r} \bigl( f(z)-f(r) \bigr)$. When $y\sim r$ there is exactly one non-zero term in the sum, $\bigl( f(r)-f(y) \bigr)$ and summing over all $y\sim r$ we get $\sum_{y\sim r}\bigl( f(r)-f(y) \bigr)$. Thus 
\small
\be\label{eq:linearization-2}
\sum_{y\in V}\sum_{z\sim y}\bigl( f(z)-f(y) \bigr) \bigl( \delta_r(z)- \delta_r(y) \bigr)=-2\Delta f(r), 
\ee
\normalsize
and (\ref{eq:linearization-1}) reduces to
\small
\be
	-2\Delta f (r)=\frac{N-2}{2N} \biggl [ \lvert V\rvert f(r)- \sum_{y\in V}f(y) \biggr ] , \label{eq:linearization-3}
\ee
\normalsize
which is equivalent to
\small
\be\label{eq:linearization-4}
\Delta f (r) = \frac{N-2}{4N} \bar{\Delta} f (r), 
\ee
\normalsize
where $\bar{\Delta}$ denotes the Laplacian for the graph completion, $\bar{G}$, of $G$. This right away implies that $\Delta f (r) =0$ and we are done. Furthermore when $G$ is a complete graph, $f$ is harmonic on $G$.
\par Now suppose $G$ is arbitrary. By equation (\ref{eq:linearization-3})
\small
\be	
\Delta \bigl(  f - \operatorname{avg}(f) \bigr)(r) = - \frac{N-2}{4N} |V| \bigl(  f - \operatorname{avg}(f) \bigr)(r). \notag
\ee
\normalsize
Now if $f$ is not constant then by the Rayleigh quotient, (\ref{eq:Rayleigh}), we see that $\lambda_1 (G) = \frac{N-2}{4N} |V| $ and $f - \operatorname{avg}(f)$ is an eigenfunction for $\lambda_1$.\\
For the "if" direction suppose for some non-constant function, $f$, that $f - \operatorname{avg}(f)$ is an eigenfunction for $\lambda_1 = \frac{N-2}{4N} |V| $. Tracing back the above computations one has (\ref{eq:linearization-1}) holds for $\phi = \delta_y$'s. Then since (\ref{eq:linearization-1}) is linear in $\phi$, one can use $f = \sum_{y\in V} f(y)\delta_y$ instead of $\phi$ which will translate to $f$ realizing $K^c_{max}$.
\end{proof}  
\section{$CCD(K,N)$ and Lower Bounds on $\lambda_1(G)$}\label{sec:lambda1}
In this section we assume that a given graph, $G$, satisfies the $CCD(K,N)$ condition. We will use the resulting global Poincar\'e inequality along with the results of the last section to find lower bounds on the first non-zero eigenvalues of such graphs. 
\begin{lemma}
Suppose $G$ satisfies $CCD(K,N)$ and $N\geq2$. Then Cheeger's isoperimetric constant, $h(G)$, satisfies
\small
\be
	h (G)  \ge  \frac{2\lvert V\rvert + 4NK +N\lvert V\rvert - 6N}{8N},\notag
\ee
\normalsize
and,
\small
\be
	\lambda_1(G) \ge \frac{\bigl(2\lvert V\rvert + 4NK +N\lvert V\rvert - 6N \bigr)^2}{128N^2 d_{max}}. \notag
\ee
\normalsize
\end{lemma}
\begin{proof}
Since $G$ satisfies the $CCD(K,N)$ condition for any $f$, we have the global Poincar\'e inequality from Theorem~\ref{thm:main-1}, 
		\small
		\be
		\sum_{y\in V}\Gamma_1(f)(y)\geq\frac{2-N}{2N}\biggl ( \sum_{y\in V}f(y)\biggr )^2+\frac{2K+\lvert V\rvert-3}{2}\sum_{y\in V}f^2(y). \notag
		\ee
		\normalsize
Suppose $F\subset V$ and $|F| \le \frac{|V|}{2}$. Let $f = \chi_F$ be the characteristic function of $F$, then (~\ref{eq:poincare}) becomes
\small
\be
	 \frac{2-N}{2N} |F|^2+\frac{2K+\lvert V\rvert-3}{2} |F|  \le 2|\partial F| . \notag
\ee
\normalsize
Hence, when $N\geq2$, 
\small
\be
	\frac{|\partial F|}{|F|} \ge \frac{2-N}{4N} |F|+\frac{2K+\lvert V\rvert-3}{4} \ge \frac{2-N}{4N} \frac{|V|}{2} + \frac{2K+\lvert V\rvert-3}{4} = \frac{2\lvert V\rvert + 4NK +N\lvert V\rvert - 6N}{8N}.  \notag
\ee
\normalsize
Now applying Cheeger's inequality (~\ref{eq:DAM}) (see~\cite{Chung-2} and~\cite{ASS}), we know that $\lambda_1(G) \ge \frac{h^2(G)}{2d_{max}}$, where $d_{max}$ is the maximum degree in the graph, $G$. Hence, 
\small
\be
	\lambda_1 \ge \frac{h^2(G)}{2d_{max}} \ge \frac{\bigl(2\lvert V\rvert + 4NK +N\lvert V\rvert - 6N \bigr)^2}{128N^2 d_{max}}. \notag
\ee
\normalsize
\end{proof}
In the rest of this section we show that any lower bound, $\lambda$, for $\lambda_1(G)$ will imply that $G$ satisfies $CCD(K,N)$ for some $K$ and $N$ (depending on $\lambda$). 
\begin{thm}
Suppose $\lambda_1(G) \ge \lambda$, then $G$ satisfies $CCD(K,N)$ for any $K$ and $N$ with
\small
\be 
N \ge \frac{2\lvert V\rvert}{|V|- \lambda},  \quad  \text{and} \quad  K \ge \frac{\lambda - |V| + 3}{2}.  \notag 
\ee
\normalsize
\end{thm}
\begin{proof}
Since $\lambda_1(G) \ge \lambda$ then by the Rayleigh quotient, (\ref{eq:Rayleigh}), we get
\small
\begin{eqnarray}\label{eq:lambda-to-poincare}
	 \sum_{y \in V} \Gamma_1 (f)(y) &\ge& \frac{\lambda}{2} \sum_{y \in V} \bigl( f - \operatorname{avg}(f)  \bigr)^2(y) = \frac{\lambda}{2} \biggl[ \sum_{y \in V} f^2(y)  +  |V|  \operatorname{avg}(f)^2 - 2  \operatorname{avg}(f) \sum_{y \in V} f(y)   \biggr] \notag \\ & = &  \frac{\lambda}{2} \sum_{y \in V} f^2(y) + \frac{\lambda}{2|V|} \bigl( \sum_{y \in V} f(y)  \bigr)^2 - \frac{\lambda}{|V|} \bigl( \sum_{y \in V} f(y)  \bigr)^2 \\ &=& \frac{\lambda}{2} \sum_{y \in V} f^2(y) - \frac{\lambda}{2|V|} \bigl( \sum_{y \in V} f(y)  \bigr)^2. \notag
\end{eqnarray}
\normalsize
Comparing (\ref{eq:lambda-to-poincare}) to the global Poincar\'e inequality, (\ref{eq:poincare}) due to the $CCD(K,N)$ condition, and one observes that $G$ satisfies $CCD(K,N)$ for any $K$, and $N$ where
\small
\be 
\frac{ \lambda }{|V|} \le\frac{ N - 2}{N},   \quad  \text{and} \quad  \lambda\le 2K+\lvert V\rvert-3 .  \notag 
\ee
\normalsize
The conclusion follows by noticing that one always have $\lambda_1(G) \le |V| $. 
\end{proof}

\setlength{\parskip}{0mm}

\end{document}